\documentclass{article}
\usepackage{geometry,amsmath,amssymb,amsthm}                % See geometry.pdf to learn the layout options. There are lots.
\usepackage{graphicx}
\usepackage{amssymb}
\usepackage{epstopdf,amsmath}

\newtheorem{prop}{Proposition}[section]
\newtheorem{lemma}[prop]{Lemma}
\newtheorem{cor}[prop]{Corollary}
\newtheorem{theorem}[prop]{Theorem}
\theoremstyle{remark}
\newtheorem{rem}[prop]{Remark}

\newcommand{\eps}{\epsilon}

\newcommand{\ZZ}{\mathbb{Z}}
\newcommand{\RR}{\mathbb{R}}
\newcommand{\covering}{\mathcal{E}}

\newtheorem*{mainThmDisc}{Theorem \ref{mainThm}$^\prime$}

\title{On the discretized sum-product problem}

\author{
Larry Guth\thanks{Massachusetts Institute of Technology, Cambridge, MA, lguth@math.mit.edu. Supported by a Simons Investigator Award.}
\and
Nets Hawk Katz\thanks{California Institute of Technology, Pasadena CA, nets@caltech.edu.
Supported by NSF grant DMS 156590.}
\and
Joshua Zahl\thanks{University of British Columbia, Vancouver, BC, jzahl@math.ubc.ca. Supported by a NSERC Discovery Grant.}
}

\begin{document}

\maketitle
 \begin{abstract}
We give a new proof of the discretized ring theorem for sets of real numbers. As a special case, we show that if $A\subset\mathbb{R}$ is a $(\delta,1/2)_1$-set in the sense of Katz and Tao, then either $A+A$ or $A.A$ must have measure at least $|A|^{1-\frac{1}{68}}$. 
\end{abstract}

\section{Introduction}
A $(\delta,\sigma)_1$-set is a discretized analog of a $\sigma$-dimensional subset of $\RR$. More precisely, for $\delta>0$ we say that a set $A\subset\RR$ is $\delta$-discretized if it is a union of closed intervals of length $\delta$. We say that a $\delta$-discretized set $A\subset [1,2]$ is a $(\delta,\sigma)_1$-set if $|A|\approx \delta^{1-\sigma}$ and if it satisfies the non-concentration condition $|A\cap I|\lessapprox |I|^\sigma|A|$ for all intervals $I$.

In \cite{KT}, Katz and Tao conjectured that a $(\delta,1/2)_1$ set cannot be approximately closed under both addition and multiplication. Specifically, they conjectured that there exists an absolute constant $c>0$ so that if $A\subset[1,2]$ is a $(\delta,1/2)_1$-set, then 
\begin{equation}\label{BourgainBound}
|A+A|+|A.A|\gtrapprox \delta^{-c}|A|.
\end{equation}

This conjecture was proved by Bourgain \cite{B1}, who established \eqref{BourgainBound} whenever $A$ is a $(\delta,\sigma)_1$ set, $0<\sigma<1$. In \cite{B1}, the constant $c=c(\sigma)>0$ is not explicitly computed, but an examination of \cite{B1} suggests that the constant is very small. Subsequent work by Bourgain-Gamburd and Bourgain \cite{BG,B2}, and work in progress by Lindenstrauss-Varj\'u and Bateman-Lie proved variants of \eqref{BourgainBound} where the set $A$ satisfies weaker non-concentration conditions. 

In this paper, we obtain a short new proof of \eqref{BourgainBound} that establishes an explicit value of $c$.

\begin{theorem}\label{mainThm}
Let $0<\sigma<1$. Let $A\subset[1,2]$ be a $\delta$-discretized set of measure $\delta^{1-\sigma}$. Suppose that for all intervals $I$, we have the non-concentration estimate
\begin{equation}\label{nonConcentrationEstimate}
|A\cap I|\leq C|I|^{\sigma}|A|.
\end{equation}
Then for every $c<\frac{\sigma(1-\sigma)}{4(7+3\sigma)}$, we have
\begin{equation}\label{largeSumProduct}
|A+A| + |A.A| \gtrsim C^{-O(1)}\delta^{-c}|A|.
\end{equation}
\end{theorem}
In \cite[Section 4]{B1}, Bourgain proved that Theorem \ref{mainThm} (with any value $c>0$) establishes the Erd\H{o}s-Volkmann ring conjecture \cite{EV}: there does not exist a measurable subring of the reals of Hausdorff dimension strictly between $0$ and $1$. This was first proved by Edgar and Miller in \cite{EM}. 

\begin{rem}
In \cite{BG}, Bourgain and Gamburd established a variant of Theorem \ref{mainThm} where the non-concentration hypothesis \eqref{nonConcentrationEstimate} was replaced by the weaker requirement
\begin{equation}\label{nonConcentrationEstimateWeaker}
|A\cap I|\leq C|I|^{\tau}|A|
\end{equation}
for some $\tau>0$ (in particular, $\tau$ may be smaller than $\sigma$). They obtained their result by reducing the general case $\tau>0$ to the special case $\tau=\sigma$, and then applying the previously established discretized sum-product theorem from \cite{B1}. A similar reduction can be used to obtain a variant of Theorem \ref{mainThm} where $A$ satisfies the weaker non-concentration hypothesis \eqref{nonConcentrationEstimateWeaker}. When this is done, the constant $c$ in \eqref{largeSumProduct} will depend on both $\sigma$ and $\tau$. 
\end{rem}

Our proof of Theorem \ref{mainThm} uses many of the ideas from Garaev's sum-product theorem in $\mathbb{F}_p$ from \cite{G}. In \cite{G}, Garaev proved that if $A\subset \mathbb{F}_p$ with $|A|<p^{7/13}(\log p)^{-4/13}$, then $|A+A|+|A.A|\gtrsim |A|^{15/14}/(\log |A|)^{2/7}$. By refining Garaev's arguments, the exponent of $15/14$ was improved to $14/13$ by Shen and the second author \cite{KS}; to $13/12$ by Bourgain and Garaev \cite{BGa}; and to $12/11$ by Rudnev \cite{Rudnev}.

 Glossing over several technical details, Garaev's proof proceeds as follows. Either $\frac{A-A}{A-A}=\mathbb{F}_p$, or there is an element $x\in \frac{A-A}{A-A}$ so that $x+1\not\in \frac{A-A}{A-A}$. If the former occurs, then Pl\"unnecke's inequality (discussed further in Section \ref{PlunneckeSection} below) implies that $|A+A|+|A.A|$ must be large. In our proof, we will call this situation the ``dense case.'' If the latter occurs, then we can write $x = \frac{a-b}{c-d}$, where $a,b,c,d\in A$. Since $x+1\not\in\frac{A-A}{A-A}$, we have $|A+(x+1)A| = |A|^2$, and this in turn implies that $|(a-b)A+(c-d)A+(c-d)A|\geq|A|^2$. Pl\"unnecke's inequality now implies that $|A+A|+|A.A|$ must be large. In our proof, we will call this situation the ``gap case.'' 

 When mimicking Garaev's argument for discretized subsets of $\RR$, we run into several issues. First, if $A\subset[1,2]$ is a $\delta$-discretized set, then the ``denominators'' in the set $\frac{A-A}{A-A}$ might be very small. Rather than considering the entire set $\frac{A-A}{A-A}$, we pick a cutoff and look at quotients where the denominator is not too small.  Several steps in the argument have to be tuned or adjusted to take account of the size of the denominator. 

 If the set of quotients is sufficiently dense, then we proceed as in the ``dense case'' of Garaev's argument. Otherwise, there is a fairly large gap in our set of quotients.  Suppose that we had an element $x$ in our set of quotients and $x+1$ was in the middle of the fairly large gap.  Then we could adapt Garaev's argument from the ``gap case'' described above. Unfortunately, there is no guarantee that this will happen.  Instead, we show that if our set of quotients has a fairly large gap, we can find a $x$ in our set of quotients so that either $x/2$ or $(x+1)/2$ is in the middle of the large gap, and we can use this element in the same way that Garaev uses $x+1$.

\subsection{Notation}

If $X\subset\RR$ and $t>0$, we will write $\covering_t(X)$ to denote the $t$-covering number of $X$, that is, the cardinality of the smallest covering of $X$ by intervals of length $t$. We will write ${\cal N}_t(X)$ to denote the closed $t$ neighborhood of $X$, and we will write $X_t$ to denote ${\cal N}_t(X)\cap t\ZZ$; this is the discretization of $X$ at scale $t$. We say a set $X\subset\RR$ is $t$-separated if every pair of points in $X$ have distance at least $t$. In particular, the set $X_t$ is always $t$-separated.

With this terminology, we will state an equivalent formulation of Theorem \ref{mainThm} that will be slightly easier to work with
\begin{mainThmDisc}
Let $0<\sigma<1$. Let $A\subset[1,2]$ be a $\delta$-separated set of cardinality $\delta^{-\sigma}$. Suppose that for all intervals $I$of length at least $\delta$, we have the non-concentration estimate
\begin{equation}\label{nonConcentrationEstimateDisc}
\#(A\cap I)\leq C|I|^{\sigma}(\#A).
\end{equation}
Then for every $c<\frac{\sigma(1-\sigma)}{4(7+3\sigma)}$, we have
\begin{equation}\label{largeSumProductDisc}
\covering_\delta(A+A) + \covering_\delta(A.A) \gtrsim C^{-O(1)}\delta^{-c}(\#A).
\end{equation}
\end{mainThmDisc}

\subsection{Thanks}
The authors would like to thank Brendan Murphy, Victor Lie, and Jianan Li for comments and corrections to a previous draft of this manuscript. The authors would also like to thank the anonymous referees for corrections and suggestions.  %, and would like to thank  for his comments on the first arxiv version.

%Observe that if $\sigma=1/2$, then \eqref{largeSumProduct} holds for all $c<\frac{1}{264}$.

\section{Preprocessing of $\delta$-separated sets with good additive properties.} \label{TreeSection}

In this section, we deal with the following technical problem. Let $X \subset [1,2]$ be a $\delta$-separated set. We will be in the situation where $X$ has good additive properties at scale $\delta$ namely,
$$\covering_\delta(X+X) \leq K (\#X).$$
We would like to know that in absolute terms, the set $X$ has the same
additive properties at each scale $t > \delta$, that is
$$\covering_t(X+X) \leq K \covering_t(X).$$
This is not necessarily true about $X$ but is true about a fairly large subset. The following lemma, which is closely related to Lemma 5.2 from \cite{BGS} will make this statement precise.

\begin{lemma} \label{Treeprocessing} Let $X \subset [1,2]$ be a $\delta$-separated subset and suppose $\#X = \delta^{-\sigma}$ for some $0<\sigma<1$.
Suppose that $$\covering_{\delta}(X+X) \leq K (\# X).$$
 Then for every $\epsilon>0$, there is a subset $\tilde X \subset X$ with 
$$\# \tilde X \gtrsim \delta^{\epsilon}(\# X),$$
so that
\begin{equation}\label{goodCoveringEveryScale}
\covering_t(\tilde X + \tilde X) \lesssim  \delta^{-10\epsilon} K \covering_t(\tilde X)\quad\textrm{for all}\ \delta<t<1,
\end{equation}
with the implicit constants depending on $\sigma$ and $\epsilon$.
\end{lemma}

\begin{proof}
We will pick $j$, a large natural number depending only on $\sigma$ and $\epsilon$. Without loss of generality, we may assume that $\delta=2^{-mj}$ for $m$ a natural number. Indeed, reducing to this case will only affect our constants by factor of $2^{O(j)}$; since $j$ depends only on $\sigma$ and $\epsilon$, this is acceptable. 

We will subdivide the interval
$[1,2]$ into the $2^j$-adic intervals. That is, for each natural number $0 \leq l \leq m$, we let
$${\cal D}_l([1,2])= \{ [1+k 2^{-lj},\ 1+(k+1) 2^{-lj}]  :  0 \leq k \leq 2^{lj} -1 \}.$$
The collection of intervals 
$${\cal D}=\bigcup_{l=0}^m  {\cal D}_l([1,2])$$
is a tree with $m+1$ levels under containment.  We will find our subset $\tilde X$ by following an algorithm that goes up the tree. We let
$X_l$ be the set of intervals in ${\cal D}_l([1,2])$ which contain a point of $X$. We now will pick subsets $Z_l$ of each $X_l$, starting with
$Z_m=X_m$. 

Next we observe that once $Z_{l+1}$ has been chosen we have that
$$\#Z_{l+1} =\sum_{I \in X_l}  \#\{J \in Z_{l+1}: J \subset I\}.$$
We know that each positive summand is between 1 and $2^j$. Thus
$$\#Z_{l+1} =\sum_{k=1}^j \sum_{\substack{I \in X_l\\ 2^{k-1} \leq \#\{J \in Z_{l+1}: J \subset I\}<2^k}} \#\{J \in Z_{l+1}: J \subset I\}.$$
We pick the value of $k$ contributing most to the sum and let 
$$Z_l=\{I \in X_l: 2^{k-1} \leq \#\{J \in Z_{l+1}: J \subset I\} < 2^k \}.$$

What we have basically done is found a large piece of $X$ which is an essentially uniform tree. We keep track of our losses. We let
$\tilde X_l$ be the set of $x \in X$ so that $x$ is in an interval of $Z_k$ for each $k >l$. We have that
$$\#\tilde X_l  \geq {1 \over (2j)^{m-l}} (\#X).$$
We observe that by making $j$ large enough, we have ensured that $(2j)^m  \leq 2^{jm \epsilon} = \delta^{-\epsilon}$. Because
a fortiori, we also have $2^m \leq \delta^{-\epsilon}$, we know that each interval in $Z_l$ contains at least $\delta^{2\epsilon} {\#X\over \#Z_l}$
elements of $X$ (recall that these intervals have length $2^{-lj}$). Thus each interval in $Z_l+Z_l$ of length $2^{-lj+1}$ contains at least $\delta^{2\epsilon} {\#X\over \#Z_l}$  elements of $X+X$ which are 
$\delta$-separated. Thus we must have
$$\#(Z_l+Z_l) \lesssim \delta^{-10 \epsilon} K (\#Z_l),$$
because otherwise $\covering_{\delta}(X+X)\gtrsim \delta^{-8\epsilon} K (\#X),$ a contradiction. We let
$$\tilde X = \bigcap \tilde X_l,$$
and the lemma is proved.

\end{proof}

\section{Pl\"unnecke's inequality and its implications}\label{PlunneckeSection}
In this section we will begin the proof of Theorem \ref{mainThm}$^\prime$. Let $A\subset[1,2]$ be a $\delta$-separated set of cardinality $\delta^{-\sigma}$ that satisfies \eqref{nonConcentrationEstimateDisc}. Fix a number $c<\frac{\sigma(1-\sigma)}{4(7+3\sigma)}$, and let $\eps>0$ be a small constant to be chosen later. In the arguments below, we will see terms of the form $\delta^{O(\eps)}$. The implicit constant will depend only on $\sigma$. 

Apply Lemma \ref{Treeprocessing} to $A$ with $\eps$ as above to obtain a set $A^\prime$ that satisfies \eqref{goodCoveringEveryScale}. Define $K$ by
$$K={\covering_\delta(A^\prime+A^\prime)+\covering_\delta(A^\prime.A^\prime) \over 
(\#A^\prime)}.$$ 
Then we have that $\covering_\delta(A^\prime+A^\prime)+\covering_\delta(A^\prime.A^\prime) \leq K(\#A^\prime)$; our goal will be to obtain a lower bound for $K$. Because
$\#(A^{\prime}) \gtrsim \delta^{O(\epsilon)} \#A,$ and, of course, 
$$A^{\prime} + A^{\prime} \subset A+A,$$
and
$$A^{\prime} . A^{\prime} \subset A . A,$$
an appropriate lower bound on $K$ implies inequality \eqref{largeSumProductDisc}.

Our next task is to find a large subset $A_1\subset A^{\prime}$ that has small expansion (in the
$\delta$-covering sense) under certain types of repeated addition and multiplication. Our main tool will be Pl\"unnecke's inequality:

\begin{prop}[Pl\"unnecke]\label{AbelianPlunnecke}
Let $G$ be an Abelian group and let $X,Y_1,\ldots,Y_k$ be subsets of $G$. Suppose that $\#(X+Y_i)\leq K_i(\#X)$ for each $i=1,\ldots,k$. Then there exists a subset $X^\prime\subset X$ so that
$$
\#(X^\prime+ Y_1 +\ldots+ Y_k) \leq \big(\Pi_{i=1}^k K_i\big) (\#X^\prime).
$$
\end{prop}

An inequality of this form was first proved by Pl\"unnecke \cite{P}. The current formulation is due to Ruzsa \cite{R}. More recently, Petridis \cite{Pe} obtained a short and elementary proof of Pl\"unnecke's inequality. 

Observe that if $X,Y_1,\ldots,Y_k$ are subsets of $G$ with $\#(X+Y_i)\leq K_i(\#X)$ for each $i=1,\ldots,k$, then whenever $X_0\subset X$ with $\#X_0\geq(\#X)/2$, we have $\#(X_0+Y_i)\leq 2K_i(\#X_0)$ for each $i=1,\ldots,k$. By repeatedly applying this observation to the set $X_0 = X\backslash X^\prime$ that is ``left over'' after applying Proposition \ref{AbelianPlunnecke}, we can obtain the following slight strengthening of Pl\"unnecke's inequality:

\begin{cor}\label{AbelianPlunneckeCor}
Let $G$ be an Abelian group and let $X,Y_1,\ldots,Y_k$ be subsets of $G$. Suppose that $\#(X+Y_i)\leq K_i(\#X)$ for each $i=1,\ldots,k$. Then there exists a subset $X^\prime\subset X$ with $\#X^\prime\geq (\#X)/2$ so that
$$
\#(X^\prime + Y_1  +  \ldots+ Y_k)\lesssim \big(\Pi_{i=1}^k K_i\big) (\#X^\prime).
$$
\end{cor}

We will also need Ruzsa's triangle inequality:

\begin{prop}[Ruzsa triangle inequality]\label{Triangle inequality} Let $G$ be an Abelian group and let $X,Y,Z \subset G$ be finite subsets. Then
$$\#(X-Z) \leq {\#(X-Y)\ \#(Y-Z) \over \#Y}.$$ \end{prop}

\begin{proof}  Let $s=x-z \in X-Z$. Then there are at least $\#Y$ distinct representations of $s$ as a sum of an element of $X-Y$ with an
element of $Y-Z$. Namely, $s= (x-y) + (y-z)$ for each $y \in Y$. \end{proof}

If $X\subset [1,2]$ is a  set, we will call a set $X^\prime$ a $\delta$-\emph{refinement} of $X$ if  $X^\prime\subset X$, and $\covering_{\delta}(X^\prime) \geq \covering_\delta(X)/2$.
We shall extend Proposition \ref{AbelianPlunnecke}, Corollary \ref{AbelianPlunneckeCor} and Proposition \ref{Triangle inequality} to the
$\delta$-covering setting, by replacing any set $X$ by $X_{\delta}$ and observing that $\#X_{\delta} \sim \covering_\delta(X)$.

\begin{cor}\label{discretizedPlunnecke}
Let $X,Y_1,\ldots,Y_k$ be  subsets of $\RR$. Suppose that 
$\covering_\delta(X+Y_i)\leq K_i\covering_\delta(X)$ for each $i=1,\ldots,k$. Then there is a $\delta$-refinement $X^\prime$ of $X$ so that 
$$
\covering_\delta(X^\prime + Y_1 + Y_2  + \ldots + Y_k)\lesssim\big(\Pi_{i=1}^k K_i\big)    \covering_\delta(X^\prime).
$$
In particular,
$$
\covering_\delta(Y_1 + Y_2  + \ldots + Y_k)\lesssim\big(\Pi_{i=1}^k K_i \big)   \covering_\delta(X).
$$
\end{cor}

\noindent To obtain Corollary \ref{discretizedPlunnecke}, replace each of the sets $X,Y_1,\ldots,Y_k$ by $X_{\delta}, (Y_1)_{\delta}, \ldots, (Y_{k})_{\delta}$. Observe that for any subset $Z$ of $X_{\delta}$, we have
$$\#\big(Z + (Y_1)_{\delta} + \ldots + (Y_{k})_{\delta}\big)  \sim \covering_\delta(Z+Y_1+ \ldots + Y_k),$$
with the implicit constant depending on $k$. Apply Corollary \ref{AbelianPlunneckeCor} to the finite sets $X_{\delta},$ $(Y_1)_{\delta}, \ldots,$ $(Y_{k})_{\delta}$, and let $R$ be the resulting refinement of $X_{\delta}$. Finally, take $X^{\prime}$ to be the set
${\cal N}_{\delta} (R) \cap X$.

In the same way, we obtain the following $\delta$-covering version of the triangle inequality.

\begin{prop} \label{discretizedTriangle}
Let $X,Y,Z$ be subsets of $\RR$. Then 
$$\covering_\delta(X-Z) \lesssim {\covering_\delta(X-Y)\ \covering_\delta(Y-Z) \over \covering_\delta(Y)}.$$ \end{prop}

We are now ready to proceed.
% We continue with the notation that for any set $Z$, the set $Z_{\delta} = Z \cap \delta \ZZ$ 
Observe that for any $x \in A^{\prime}$, we have that
$$\#\big((x A^{\prime})_\delta\big) \sim \covering_\delta(x A^{\prime}).$$
First, by Cauchy-Schwarz, the condition $\covering_\delta(A^{\prime}.A^{\prime}) \leq K(\#A^{\prime})$ implies that

$$ \sum_{x,y \in A^{\prime}} \#\big((xA^{\prime})_{\delta} \cap (yA^{\prime})_{\delta}\big)
\gtrsim (\#A^{\prime})^3 K^{-1}. $$

\noindent Select an element $b \in A^{\prime}$ so that

\begin{equation}\label{popularB}
\sum_{x \in A^{\prime}}  \#\big((xA^{\prime})_{\delta} \cap (bA^{\prime})_{\delta}\big)
\gtrsim (\#A^{\prime})^2 K^{-1}.
\end{equation}

\noindent By dyadic pigeonholing, we can select a set $\bar A \subset A^{\prime}$ and a number $K^{-1} \le \rho \le 1$ with $\#\bar A \gtrsim (\log K) \rho (\#A^{\prime})\gtrsim |\log\delta|^{-1} \rho (\#A^{\prime})$ so that

$$ \#\big( (aA^{\prime})_{\delta} \cap (b A^{\prime})_{\delta}\big)\sim (\#A^{\prime}) K^{-1} \rho^{-1} \quad\ \textrm{for each}\ a\in \bar A.$$

At the end of our argument, we will see that the worst-case occurs when $\rho=1$. Thus a casual reader may safely set $\rho=1$. 
We will consider the $\delta$-covering number
$\covering_\delta(a A^{\prime} \pm b A^{\prime})$ when $a \in \bar A$.  Let $X =( aA^{\prime})_{\delta} \cap (bA^{\prime})_{\delta}$, so 
$\#X \ge K^{-1} \rho^{-1} (\#A^{\prime})$. We have

\begin{equation*}
\begin{split}
\covering_\delta(X + a A^{\prime}) & \lesssim K (\# A^{\prime}) \leq K^2 \rho (\#X),\\
\covering_\delta(X + b A^{\prime}) & \lesssim K (\#A^{\prime}) \leq K^2 \rho (\#X).
\end{split}
\end{equation*}

\noindent By Corollary \ref{discretizedPlunnecke}, we obtain

\begin{equation}\label{aApbA} 
\covering_\delta\big(aA^{\prime} + bA^{\prime}\big) \lesssim K^4 \rho^2 (\#X) \lesssim K^3 \rho (\#A^{\prime}).
\end{equation}

\noindent By Proposition \ref{discretizedTriangle}, substituting $aA^{\prime}$ for $X$, $b A^{\prime}$ for $Z$ and $-X$ for $Y$, we obtain

\begin{equation}\label{aAmbA} 
\covering_\delta\big(aA^{\prime} - bA^{\prime}\big) \lesssim K^4 \rho^2 (\#X) \lesssim K^3 \rho (\#A^{\prime}).
\end{equation}

\noindent Combining the results (\ref{aApbA}) and (\ref{aAmbA}), we obtain

\begin{equation}\label{aApmbA} 
\covering_\delta\big(aA^{\prime} \pm bA^{\prime}\big)\lesssim K^4 \rho^2 (\#X) \lesssim K^3 \rho(\#A^{\prime}).
\end{equation}

So far elements of $\bar A$ have good properties for multiplying $A^{\prime}$ and adding such
dilates, but we would like to take advantage as well of the additive properties of $\bar A$. To wit,
we have that $\bar A + \bar A \subset A^{\prime} + A^{\prime}$ so that
$$\covering_\delta(\bar A+ \bar A) \leq \covering_\delta(A^{\prime} + A^{\prime}) \leq K (\#A^{\prime})
\leq K |\log\delta|\rho^{-1} (\#\bar A). $$

However, we might have liked to have estimates on $\covering_t(\bar A + \bar A)$ for $t > \delta$. To
fulfill our desires, we apply Lemma \ref{Treeprocessing} to $\bar A$ with $\eps$ as above to obtain a set $A_1=\tilde {\bar A}.$
We immediately obtain the estimate
$$\covering_t(A_1+A_1) \leq K |\log\delta| \rho^{-1} \delta^{-O(\epsilon)} \covering_t(A_1)\quad\textrm{for all}\ \delta<t<1.$$

In the next two lemmas, we will obtain estimates for the cardinality of various sums involving $A_1$ and $A^{\prime}$. These lemmas constitute an ingredient of our argument that does not occur in Garaev's finite field argument. We lose some resolution
when we multiply $A^{\prime}$ by a number much smaller than 1 and then calculate $\delta$-covering number, but we win some of this
back by considering nonconcentration for $A^{\prime}$.

\begin{lemma}\label{sizeOfd1d2A}
Let $d_1 = a_1-b_1,\ d_2 = a_2-b_2$, with $a_1,a_2,b_1,b_2\in A_1$. Then
\begin{equation}\label{d1A1d2A2}
\covering_\delta(d_1A^{\prime}+d_2A^{\prime})\lesssim C\delta^{-\eps}K^{12}\rho^4\max(|d_1|,|d_2|)^{\sigma}(\#A^{\prime}).
\end{equation}
\end{lemma}
\begin{proof}
By Corollary \ref{discretizedPlunnecke} and \eqref{aApmbA}, there is a refinement $A^{\prime \prime}$ of $A^{\prime}$ so that 
\begin{equation}\label{bd1Ad2A}
\covering_\delta(bA^{\prime \prime}+d_1A^{\prime}+d_2A^{\prime}) \leq \covering_\delta(bA^{\prime \prime}+
a_1A^{\prime}-b_1A^{\prime}+a_2A^{\prime}- b_2 A^{\prime}) \lesssim K^{12}\rho^4(\#A^{\prime}),
\end{equation}
where $b\in A^{\prime}$ is the element satisfying \eqref{popularB}.

Next, observe that $d_1A^{\prime}+ d_2A^{\prime}$ is contained in an interval of length $d \lesssim \max(|d_1|,|d_2|)$. Thus
$$
\covering_\delta(d_1A^{\prime} + d_2A^{\prime})\lesssim \big(\covering_d(bA^{\prime \prime})\big)^{-1}\ \covering_\delta(bA^{\prime \prime}+d_1A^{\prime} + d_2A^{\prime})
\lesssim \big(\covering_d(bA^{\prime \prime})\big)^{-1} K^{12}\rho^4 (\#A^{\prime}).
$$

Since $\#A^{\prime \prime}  \sim \#A^{\prime}$ and $A^{\prime}$ satisfies the non-concentration estimate 
$\#(A^{\prime} \cap I)\lesssim C\delta^{-\eps}d^{\sigma}(\#A^{\prime})$ for each interval $I$ of length $d$, following from the
comparison of the size of $A^{\prime}$ to that of $A$ and the non-concentration estimate for $A$ given in 
(\ref{nonConcentrationEstimateDisc}), we must have that 
$$\covering_d(b A^{\prime \prime})  \gtrsim \frac{\#A^{\prime \prime}}{C \delta^{-\eps}d^{\sigma}(\#A^{\prime})}=C^{-1}\delta^{\eps} d^{-\sigma}.$$ 
This completes the proof of the lemma.
\end{proof}

\begin{lemma}\label{sizeOfd1d2AIterated}
Let $d_1 = a_1-b_1,\ d_2 = a_2-b_2$, with $a_1,a_2,b_1,b_2\in A_1$. Then for each $k\geq 2$, there is a  set $A_2\subset A_1$ with $\#A_2\geq(\#A_1)/4$ so that
\begin{equation}\label{d1A2d2kA2}
\covering_\delta\big(d_1A_2+ \underbrace{d_2A_2 +\ldots + d_2A_2}_{\text{$k$ times}}\big) \lesssim C \delta^{-O(\eps)} 
|\log\delta|^{O(1)}K^{11+k}\rho^{5-k} \max(|d_1|,|d_2|)^{\sigma}(\#A^{\prime}). 
\end{equation}
The implicit constant in the $|\log\delta|^{O(1)}$ and $\delta^{-O(\eps)}$ terms depend on $k$.
\end{lemma}
\begin{proof}
As in the proof of Lemma \ref{sizeOfd1d2A}, select a refinement $A^{\prime \prime}$ of $A^{\prime}$ so that 
$\covering_\delta(bA^{\prime \prime}+d_1A^{\prime}+d_2A^{\prime}) \lesssim K^{12}\rho^4(\#A^{\prime})$. This implies that
$$
\covering_\delta(bA^{\prime \prime}+d_1A_1+d_2A_1) \lesssim \delta^{-O(\epsilon)} K^{12}\rho^3(\#A_1).
$$
 We have that 
 $$\covering_\delta(d_2A_1+d_2A_1) =\covering_{{\delta \over d_2}}(A_1+A_1) \leq K \delta^{-O(\epsilon)}
 \rho^{-1}|\log\delta|\ \covering_\delta(d_2 A_1),$$ 
 and thus by Corollary \ref{discretizedPlunnecke}, there is a refinement $A_1^\prime$ of $A_1$ so that the $k$-fold sum 
 $$\covering_\delta(d_2A_1^\prime+\ldots+d_2A_1^\prime) \leq K^{k-1}\delta^{-O(\epsilon))}\rho^{1-k}|\log\delta|^{k-1}\covering_\delta(d_2A_1^\prime).$$
 Apply Corollary \ref{discretizedPlunnecke} with $X= d_2A_1^\prime$, $Y_1 = d_2A_1^\prime+\ldots+d_2A_1^{\prime}$ (this is a $(k-1)$-fold sum), and $Y_2 = bA^{\prime \prime}+d_1A_1$. We conclude that there is a refinement $A_2$ of $A_1^\prime$ so that 
\begin{equation*}
\begin{split}
\covering_\delta(bA^{\prime \prime}+d_1A_2+d_2A_2+\ldots+d_2A_2) & \lesssim (K^{12}\rho^3)\delta^{-O(\epsilon)}(K^{k-1}\rho^{1-k}|\log\delta|^{k-1}) (\#A_1)\\
& =K^{11+k}\rho^{4-k} \delta^{-O(\epsilon)}   |\log\delta|^{k-1} (\#A_1).
\end{split}
\end{equation*}
We now proceed as in the proof of Lemma \ref{sizeOfd1d2A}. Observe that $d_1A_2 + d_2A_2 + \ldots\ + d_2A$ is contained in an interval of length $d \lesssim (k+1)\max(|d_1|,|d_2|)$. Thus
\begin{equation*}
\begin{split}
\covering_\delta(d_1A_2 + d_2A_2 +\ldots + d_2A_2)&\lesssim \big(\covering_d(bA^{\prime \prime})\big)^{-1}\ \covering_\delta(bA^{\prime \prime}+d_1A_2+ d_2A_2+\ldots+d_2A_2)\\
&\lesssim \big(\covering_d(bA^{\prime \prime})\big)^{-1}K^{11+k}\rho^{4-k} \delta^{-O(\epsilon)} |\log\delta|^{O(1)}(\#A_1).
\end{split}
\end{equation*}

Since $\#A^{\prime \prime}  \sim \#A^{\prime}$ and 
$A^{\prime}$ satisfies the non-concentration estimate $\#(A^{\prime} \cap I)\leq C\delta^{-\eps} d^{\sigma}(\#A^{\prime})$ for each interval $I$ of length $d$, we must have that 
$\covering_d(b A^{\prime \prime}) \gtrsim \frac{\#A^{\prime \prime}}{C \delta^{-\eps}d^{\sigma}(\#A^{\prime})}=C^{-1}\delta^{\eps} d^{-\sigma}$. Since $\#A^\prime\geq\delta^{\eps}\#A_1$, we have
\begin{equation*}
\begin{split}
\covering_\delta(d_1A_2 + d_2A_2 +\ldots + d_2A_2) & \lesssim C \delta^{-O(\epsilon)}|\log\delta|^{O(1)}d^\sigma K^{11+k}\rho^{4-k} (\#A^\prime).\qedhere
\end{split}
\end{equation*}
\end{proof}

\section{The structure of $\frac{A-A}{A-A}$: dense versus gap cases}
Let $A_1$ be the set constructed in the previous section, and let $\gamma\in (0,1/2)$ be a parameter we will specify later.  Consider the set

$$ 
B = \Big\{ \frac{a_1 - a_2}{a_3 - a_4} : a_i \in A_1, |a_3 - a_4| > \delta^\gamma \Big\}. 
$$

\noindent Since $A_1\subset [1,2]$, we have that $B\subset [-\delta^{-\gamma},\delta^{-\gamma}].$ We also have $0,1\in B$. Choose a positive integer $m$ so that $2^{-m}\sim\delta^{1-2\gamma}$. Define $s= 2^{-m}$. 

%The next lemma establishes a dichotomy between
%
%
\begin{lemma}\label{denseVsGapLemma}
At least one of the following two things must happen.
\begin{itemize}
\item[(A):] There exists a point $b\in B\cap [0,1]$ with
$$
\max\Big(\operatorname{dist}(b/2,\ B),\ \ \operatorname{dist}\Big(\frac{b+1}{2},\ B\Big)\Big)\geq s.
$$
\item[(B):] $\mathcal{E}_s(B\cap [0,1]) \gtrsim s^{-1}.$
\end{itemize}
\end{lemma}
\begin{proof}
% Let 
% $$
% \tilde B  = \{x\in [0,1]\cap 2s\ZZ\colon \operatorname{dist}(x, B)\leq 2s\}.
% $$

% $\tilde B$ is a discretization of $B\cap [0,1]$ at scale $2s$, and $\#\tilde B\lesssim \mathcal{E}_s(B)$. For each $\tilde b\in\tilde B\cap 4s\ZZ$, there is a point $b\in B$ with $|b-\tilde b|\leq 2s$. 
%
Suppose that Item (A) does not occur. Let $\tilde b\in B_{2s}$ and let $b\in B$ be the corresponding point with $|b-\tilde b|\leq 2s$. Then there is an element $b^\prime\in B$ with $|b^\prime-b/2|<s$, and thus 
$$
|b^\prime-\tilde b/2|\leq |b^\prime-b/2| + |b/2-\tilde b/2| < s+s=2s,
$$
so $\tilde b/2\in B_{2s}$. 

Similarly, there is an element $b^{\prime\prime}\in B$ with $|b^{\prime\prime}-\frac{b+1}{2}|<s$, and thus 
$$
\Big|b^{\prime\prime}-\frac{\tilde b+1}{2} \Big|\leq \Big|b^{\prime\prime}-\frac{b+1}{2}\Big| + \Big|\frac{b+1}{2}-\frac{\tilde b+1}{2}\Big| < s+s=2s,
$$
so $\frac{\tilde b+1}{2}\in B_{2s}$. 

We will now prove by induction that for each $ n = 1,\ldots,m-1,$ every dyadic rational of the form $p/2^n \in[0,1]$ is contained in $B_{2s}$. Indeed, if $n=0$ then the result holds since $0,1\in B$ implies that $0,1\in B_{2s}$. Now suppose the result has been proved for some value of $n\leq m-2$, and let $p/2^{n+1}\in[0,1]$. If $p<2^n$, then by the induction hypothesis $p/2^{n}\in B_{2s}$, and thus $\frac{1}{2}p/2^{n}=p/2^{n+1}\in B_{2s}$. If $p\geq 2^n$, then by the induction hypothesis, $(p-2^n)/2^{n}=p/2^n-1\in B_{2s}$, and thus $p/2^{n+1} = \frac{p/2^n-1}{2}+\frac{1}{2}\in B_{2s}$. This completes the induction.  We conclude that Item (B) holds. 
\end{proof}

We say we are in the \emph{gap case} if Item (A) holds and in the \emph{dense case} if Item (B) holds. Note that these are not mutually exclusive.
\section{The dense case}
By pigeonholing, we can select $b_1, b_2, b_3, b_4 \in A_1$ with $|b_3 - b_4| > \delta^\gamma$ and $|b_1-b_2|\leq |b_3-b_4|$ so that

\begin{equation} \label{closebsparse} \#\Big\{ (a_1, ..., a_4) \in A_1^4\colon \Big| \frac{a_1 - a_2}{a_3 - a_4} - \frac{b_1 - b_2}{b_3 - b_4}\Big| < \delta^{1 - 2\gamma},\  |a_3 - a_4| > \delta^{\gamma} \Big\}  \lesssim 
(\#A_1)^4 \delta^{1-2\gamma}.
\end{equation}

By Lemma \ref{sizeOfd1d2A}, we have

\begin{equation}\label{upperBoundOnd1Ad2A}
\covering_\delta\big((b_1 - b_2) A_1 + (b_3 - b_4) A_1\big) \leq \covering_\delta\big((b_1-b_2)A^{\prime}+(b_3-b_4)A^{\prime}\big) \lesssim  C \delta^{-\eps}K^{12} \rho^4 |b_3-b_4|^\sigma (\#A^\prime).
\end{equation}

We will now establish a lower bound on $\covering_\delta\big((b_1 - b_2) A_1 + (b_3 - b_4) A_1\big)$. Define $Q \subset A_1^4$ to be the set of quadruples obeying

\begin{equation} \label{quad} (b_3 - b_4) a_1 + (b_1 - b_2) a_4 = (b_3 - b_4) a_2 + (b_1 - b_2) a_3 + O(\delta). \end{equation}
\noindent Cauchy-Schwarz gives $ \covering_\delta \big(  (b_1 - b_2) A_1+(b_3 - b_4) A_1\big)\gtrsim  (\#A_1)^4 / (\#Q)$, so our goal is now to find an upper bound for $Q$.  Note that \eqref{quad} implies that

$$ a_1 + \frac{b_1 - b_2}{b_3 - b_4} a_4 = a_2 + \frac{b_1 - b_2}{b_3 - b_4} a_3 + O(\delta |b_3 - b_4|^{-1}), $$

which implies that

\begin{equation}\label{smallDifference} 
\left| \frac{a_1 - a_2}{a_3 - a_4} - \frac{b_1 - b_2}{b_3 - b_4} \right| \lesssim \delta |b_3 - b_4|^{-1} |a_2 - a_4|^{-1}.
\end{equation}

We first consider quadruples $(a_1, ..., a_4) \in Q$ where $|a_3 - a_4| \ge \delta^{\gamma}$.  For each such quadruple, \eqref{smallDifference} implies that

$$ \left| \frac{a_1 - a_2}{a_3 - a_4} - \frac{b_1 - b_2}{b_3 - b_4} \right| \lesssim \delta^{1 - 2 \gamma}. $$

\noindent Comparing with \eqref{closebsparse}, we see that the number of such quadruples is $\lesssim (\#A_1)^4 \delta^{1- 2 \gamma}\lesssim  (\#A_1)^3 \delta^{(1-\sigma)-2\gamma-O(\epsilon)}$. Thus if at least half the quadruples from $Q$ are of this form, then 
\begin{equation*}
\covering_\delta\big( (b_1 - b_2)A_1 + (b_3-b_4)A_1\big) \gtrsim (\#A_1)\delta^{2\gamma+\sigma-1+O(\epsilon)}\gtrsim \rho|\log\delta|^{-1}\delta^{2\gamma+\sigma-1} (\#A^{\prime}).
\end{equation*}
By \eqref{upperBoundOnd1Ad2A}, we conclude that if at least half the quadruples from $Q$ are of this form, then
\begin{equation}\label{bound1OnK}
K \gtrsim   (C|\log\delta|)^{-O(1)} \rho^{-1/4} \delta^{\frac{2\gamma+\sigma-1}{12}+O(\epsilon)}\gtrsim (C|\log\delta|)^{-O(1)} \delta^{\frac{2\gamma+\sigma-1}{12}+O(\epsilon)}.
\end{equation} 

On the other hand, we consider quadruples  $(a_1, ..., a_4) \in Q$ where $|a_3 - a_4| \le \delta^{\gamma}$. We begin by choosing elements $a_1,a_4\in A_1$. By our non-concentration hypothesis \eqref{nonConcentrationEstimate} and the requirement $|a_3-a_4|\leq\delta^{\gamma}$, the number of admissible $a_3$ is at most $C\delta^{\gamma\sigma}(\#A)$. Next, $a_2$ must lie in an interval of length $\leq \delta|b_3-b_4|^{-1}$. By our non-concentration hypothesis, the number of admissible $a_2$ is at most $C \delta^{\sigma}(\#A)|b_3-b_4|^{-\sigma}=C|b_3-b_4|^{-\sigma}$. Thus the set of quadruples of this type has size at most $\big(\#A_1\big)^2\big(C\delta^{\gamma\sigma}(\#A)\big)\big(C|b_3-b_4|^{-\sigma}\big)$. Thus if at least half the quadruples from $Q$ are of this form, then 
\begin{equation*}
\begin{split}
\covering_\delta\big((b_1 - b_2)A_1 + (b_3-b_4)A_1\big)&\gtrsim  \frac{(\#A_1)^4}{\big(\#A_1\big)^2\big(C\delta^{\gamma\sigma}(\#A)\big)\big(C|b_3-b_4|^{-\sigma}\big)}\\
&\gtrsim  C^{-O(1)}\rho^2|\log\delta|^{-1}\delta^{-\gamma\sigma+O(\epsilon)}|b_3-b_4|^{\sigma}(\#A^{\prime}).
\end{split}
\end{equation*}
By \eqref{upperBoundOnd1Ad2A}, we conclude that if at least half the quadruples from $Q$ are of this form, then
\begin{equation}\label{bound2OnK}
K \gtrsim  (C|\log\delta|)^{-O(1)} \rho^{-1/6} \delta^{\frac{-\gamma\sigma}{12}+O(\epsilon)}\gtrsim (C|\log\delta|)^{-O(1)}  \delta^{\frac{-\gamma\sigma}{12}+O(\epsilon)}.
\end{equation} 
If we are in the dense case, then 
\begin{equation}\label{gapCaseBound}
K\gtrsim \delta^{O(\epsilon)}(C|\log\delta|)^{-O(1)}\min\big(\delta^{\frac{2\gamma+\sigma-1}{12}},\ \delta^{\frac{-\gamma\sigma}{12}}   \big).
\end{equation}

\section{The gap case}
In this section, we will suppose that we are in the gap case. This means that there exists $b=\frac{b_1-b_2}{b_3-b_4}\in B\cap [0,1]$ so that either (A.1): $b/2$ is at least $s$-separated from $B$ or (A.2): $\frac{b+1}{2}$ is at least $s$-separated from $B$ (recall $s \sim \delta^{1-2\gamma}$). The reader should recall that by the definition of $B$, we have that $|b_3-b_4| \geq \delta^{\gamma}$.
In Case (A.1), write $b/2=e_1/e_2$, while in Case (A.2), write $\frac{b+1}{2}$ as $e_1/e_2$. In Case (A.1) we can write $e_1 = d_1,\ e_2 = d_2+d_2$, where $d_1,d_2\in A_1-A_1$. In Case (A.2) we can write $e_1 = d_1 + d_2,\ e_2 = d_2+d_2$, where $d_1,d_2\in A_1-A_1$. 

We will prove a lower bound on $\covering_\delta(e_1 A_1 + e_2 A_1)$. Define $Q \subset A_1^4$ to be the set of quadruples obeying

\begin{equation} \label{quads} 
e_2 a_1 + e_1 a_4 = e_2 a_2 + e_1 a_3 + O(\delta). 
\end{equation}

Cauchy-Schwarz gives $\covering_\delta(e_1 A_1 + e_2 A_1) \gtrsim (\#A_1)^4 / (\#Q)$, so our goal is now to find an upper bound for $Q$.  Note that (\ref{quads}) implies that

\begin{equation}\label{re-arranged_quads} 
a_1 + \frac{e_1}{e_2} a_4 = a_2 + \frac{e_1}{e_2} a_3 + O\big(\delta |e_2|^{-1}\big),
\end{equation}
which implies that

\begin{equation}\label{farFromD1D2}
\left| \frac{a_1 - a_2}{a_3 - a_4} - \frac{e_1}{e_2} \right| \lesssim \delta |e_2|^{-1} |a_2 - a_4|^{-1}.
\end{equation}

We first consider quadruples $(a_1, ..., a_4) \in Q$ where $|a_3 - a_4| \ge \delta^\gamma$.  For each such quadruple, \eqref{farFromD1D2} implies that

$$ \left| \frac{a_1 - a_2}{a_3 - a_4} - \frac{e_1}{e_2} \right| \lesssim \delta^{1 - 2 \gamma}. $$
Now since we are in the gap case, $e_1/e_2$ is at least $s \sim \delta^{1-2\gamma}$ separated from $B$. But since $|a_3 - a_4| \geq \delta^\gamma$, $ \frac{a_1 - a_2}{a_3 - a_4} \in B$.  This shows that there are no quadruples in $Q$ with $|a_3 - a_4| \geq \delta^\gamma$.  

We conclude that every quadruple in $Q$ has $|a_3 - a_4| \leq \delta^\gamma$. Select elements $a_1,a_4\in A_1$. By the non-concentration hypothesis \eqref{nonConcentrationEstimate}, the set of admissible $a_3$ has size at most $C\delta^{\gamma\sigma}(\#A)$. Finally, $a_2$ must lie in an interval of length $\delta|e_2|^{-1}$; again by the non-concentration hypothesis, the set of admissible $a_2$ has size at most  $C(\delta/|e_2|)^{\sigma}(\#A^{\prime})= C |e_2|^{-\sigma}$. All together, we have

$$ |Q| \lesssim C^{O(1)} (\#A_1)^2 (\#A^{\prime}) \delta^{\gamma\sigma} |e_2|^{-\sigma}. $$
This gives us the lower bound

\begin{equation}\label{lowerBoundOnd1d2A}
\covering_\delta(e_1 A_1 + e_2 A_1) \gtrsim C^{-O(1)}\rho^2 \delta^{O(\epsilon)}|\log\delta|^{-2}\delta^{-\gamma\sigma}|e_2|^{\sigma} (\#A^{\prime}).
\end{equation}

Note that nothing in the argument obtaining this lower bound would change if we replaced $A_1$ by a refinement of $A_1$. We conclude that 

\begin{equation}\label{lowerBoundOnd1d2A2}
\covering_\delta(e_1 A_2 + e_2 A_2) \gtrsim C^{-O(1)}\rho^2 \delta^{O(\epsilon)} |\log\delta|^{-2}\delta^{-\gamma\sigma}|e_2|^{\sigma} (\#A^{\prime})
\end{equation}
whenever $A_2$ is a refinement of $A_1$. Applying Lemma \ref{sizeOfd1d2AIterated} with $d_1 = b_1-b_2$, $d_2 = b_3-b_4$, and $k=2$ (in Case (A.1)) and $k=3$ (in Case (A.2), which is worse), we obtain the bound
\begin{equation}\label{upperBoundOnd1d2A}
\covering_\delta(e_1 A_2 + e_2 A_2)\lesssim C \delta^{-O(\epsilon)} |\log\delta|^{O(1)}K^{14}\rho^2 |e_2|^\sigma(\#A^{\prime}).
\end{equation}

Combining \eqref{lowerBoundOnd1d2A} and \eqref{upperBoundOnd1d2A}, we have
$$
C^{-O(1)}\rho^2 \delta^{O(\epsilon)} |\log\delta|^{-1}\delta^{-\gamma\sigma}|e_2|^{\sigma}(\#A^{\prime})\lesssim C \delta^{-O(\epsilon)} |\log\delta|^{O(1)}K^{14}\rho^2 |e_2|^{\sigma}(\#A^{\prime}),
$$
and thus

\begin{equation}\label{KBoundGap}
K \gtrsim \delta^{O(\epsilon)} (C|\log\delta|)^{-O(1)}\delta^{\frac{-\gamma\sigma}{14}}.
\end{equation}

By Lemma \ref{denseVsGapLemma}, at least one of \eqref{gapCaseBound} or \eqref{KBoundGap} must hold. Selecting $\gamma = \frac{7(1-\sigma)}{2(7+3\sigma)},$ we conclude that 

$$
K \gtrsim \delta^{O(\epsilon)} (C|\log\delta|)^{O(1)}\delta^{-\frac{\sigma(1-\sigma)}{4(7+3\sigma)}}.
$$
Since
$$
\covering_\delta(A^\prime+A^\prime)+\covering_\delta(A^\prime.A^\prime) \leq \covering_\delta(A+A)+\covering_\delta(A.A),
$$
we conclude that
$$
\covering_\delta(A+A)+\covering_\delta(A.A)\gtrsim \delta^{O(\epsilon)} (C|\log\delta|)^{O(1)}\delta^{-\frac{\sigma(1-\sigma)}{4(7+3\sigma)}} (\#A).
$$
Thus if $\eps>0$ is selected sufficiently small (depending only on $c$ and $\sigma$), then \eqref{largeSumProductDisc} holds. 

% \section{Remarks}\label{remarksSection}
% Our proof of Theorem \ref{mainThm} only uses the non-concentration estimate \ref{nonConcentrationEstimate} at the two scales $|I|=\delta^{\gamma}$ and $|I| = \delta^{1-\gamma}$. This is in contrast to the previous theorems of Bourgain \cite{B1, B2}, which exploit non-concentration at many scales. 

\end{document}